\documentclass[12pt,a4paper]{amsart}
\usepackage{amssymb,amsmath}

\headsep=1cm \numberwithin{equation}{section}
\newtheorem{theorem}{Theorem}[section]

\newtheorem{proposition}[theorem]{Proposition}
\newtheorem{corollary}[theorem]{Corollary}

\theoremstyle{definition}
\newtheorem{definition}[theorem]{Definition}
\theoremstyle{remark}
\newtheorem{remark}[theorem]{Remark}
\newtheorem{example}[theorem]{Example}

\newcommand\Spec{\operatorname{Spec}}

\newcommand{\SR}{\operatorname{s-rad}}
\newcommand{\height}{\operatorname{s-ht}}

\newcommand{\OS}{\operatorname{S-Spec}}

\newcommand{\fp}{\frak{p}}
\newcommand{\fq}{\frak{q}}

\begin{document}

\author[]{A. R. Naghipour}
\title[]
{Strongly Prime Submodules}
\address{A.R. Naghipour, Department of Mathematics, Shahrekord
University, P.O.Box 115, Shahrekord, Iran}
\email{Naghipour@sci.sku.ac.ir} 
\thanks{The author is supported by Shahrekord
University}

\begin{abstract}
Let $R$ be a commutative ring with identity. For an $R$-module
$M$, the notion of strongly prime submodule of $M$ is defined. It
is shown that this notion of prime submodule inherits most of the
essential properties of the usual notion of prime ideal. In
particular, the Generalized Principal Ideal Theorem is extended
to modules.
\end{abstract}

\maketitle

\section*{0. Introduction }
Throughout this paper all rings are commutative with identity and
all modules are unitary. Also we consider $R$ to be a ring and
$M$ a unitary $R$-module.

For a submodule $N$ of $M$, let $(N : M)$ denote the set of all
elements $r$ in $R$ such that $rM\subseteq N$. Note that $(N :
M)$ is an ideal of $R$, in fact, $(N : M)$ is the annihilator of
the $R$-module $M/N$. A proper submodule $N$ of $M$ is called
{\it prime} if $rx\in N$, for $r\in R$ and $x\in M$, implies that
either $x\in N$ or $r\in(N :M)$. This notion of prime submodule
was first introduced and systematically studied in Dauns (1978)
and recently has received a good deal of attention from several
authors, see for example Man and Smith (2002), McCasland and
Smith (1993), McCasland et al. (1997) and Moore and Smith (2002).

In this article, we introduce a slightly different notion of prime
submodule and call it strongly prime submodule. First of all, we
bring a notation.

 \noindent {\bf Notation.} Let $N$ be a submodule of $M$ and let $x\in M$. We
  denote the ideal $(N+Rx :M)$ by
$I_{x}^{N}$. Therefore, $I_{x}^{N}=\{r\in R|rM\subseteq N+Rx\}$.

Let $P$ be a proper submodule of $M$. We say that $P$ is a {\it
strongly prime} submodule if $I_{x}^{P}y\subseteq P$, for $x,y\in
M$, implies that either $x\in P$ or $y\in P$. We call a proper
submodule $C$ of $M$ to be a {\it strongly semiprime} submodule if
$I_{x}^{C}x\subseteq C$, for $x\in M$, implies that $x\in C$.

Note that if we consider $R$ as an $R$-module, then strongly
prime (respectively, semiprime) submodules are exactly prime
(respectively, semiprime) ideals of $R$.

Our definition of strongly prime (respectively, semiprime)
submodule seems more natural, comparing to the usual notion of
prime (respectively, semiprime) ideal of a ring. We will show that
every strongly semiprime submodule of $M$ is an intersection of
strongly prime submodules. Note that this result is not
 true for semiprime submodules, see Jenkins and Smith (1992).

This article consists of two sections. In the first section we
prove some preliminary facts about strongly prime submodules,
which one could expect. In Section 2, as an application of our
result in Section 1, we state and prove a module version of the
Generalized Principal Ideal Theorem.

\vspace{.1in}
\section{Strongly Prime Submodules}
We begin with the following proposition.
\begin{proposition} Let $M$ be an $R$-module. Then the following hold.\\
(1) Any strongly prime submodule of $M$ is prime.\\
(2) Any maximal submodule of $M$ is strongly prime.
\end{proposition}
\begin{proof}
(1) Suppose on the contrary that $P$ is not a prime submodule.
Then there exist $x\in M\setminus P$ and $r\in R$ such that
$rx\in P$ and $rM\nsubseteq P$. So there exits $y\in M$ such that
$ry\not\in P$. We have
$$I_{x}^{P}ry=rI_{x}^{P}y\subseteq r(P+Rx)\subseteq P.$$
Since $P$ is a strongly prime submodule, we should have $x\in P$
or $ry\in P$,
which is a contradiction.\\
(2) Let $x,y\in M$ and $I_{x}^{P}y\subseteq P$. If $x\not\in P$,
then $P+Rx=M$ and hence $I_{x}^{P}=R$. It follows that $y\in P$,
which completes the proof.
\end{proof}

Before we continue, let us show that a prime submodule need not
be a strongly prime (or even a strongly semiprime) submodule.
\begin{example}
Let $R$ be a ring and ${\fp}\in{\Spec}(R)$. Then $({\fp},{\fp})$
is a prime submodule of the $R$-module $R\times R$. But it is not
a strongly prime (or strongly semiprime) submodule because
$I_{(1,0)}^{({\fp},{\fp})}(1,0)\subseteq {\fp}(1,0)\subseteq
({\fp},{\fp})$, and $(1,0)\not\in({\fp},{\fp})$.
\end{example}

 \noindent {\bf Notation.} The set of all strongly prime
 submodules
of $M$ is denoted by ${\OS}_{R}(M)$.

\begin{proposition}
Let $V$ be a vector space over a field $F$. Then
$${\OS}_{F}(V)=\{W|W\;\;{\mbox{is a maximal subspace of}}\;\; V\}.$$
\end{proposition}
\begin{proof}
By the above proposition, every maximal subspace is strongly
prime. For the converse, suppose to the contrary that $W$ is a
strongly prime subspace of $V$ which is not a maximal subspace.
Then there exists $x\in V\setminus W$ such that $Fx+W\neq V$. For
any $y\in M$, we have
$$I_{x}^{W}y=\{r\in F|rV\subseteq Fx+W\}y=\{0\}y=\{0\}\subseteq W.$$
It follows that $y\in W$ and hence $W=V$, which is a
contradiction. Thus every strongly prime subspace is maximal.
\end{proof}
 Following Dauns (1980), we say that a proper submodule $N$ of an $R$-module
  $M$ is {\it semiprime} if whenever $r^{2}x\in N$, where $r\in R$
 and $x\in M$, then $rx\in N$.
The ring $R$ is called {\it Max-ring} if every $R$-module has a
maximal submodule. Max-Rings, which also called {\it $B$-rings},
were introduced by Hamsher (1967) and has been studied by several
authors, see for example Camillo (1975), Faith (1973, 1995),
Hirano (1998) and Koifmann (1970).

The following corollary provides characterizations of Max-rings.
\begin{corollary}
Let $R$ be a ring. Then the following are equivalent.

(1) $R$ is Max-ring.

(2) Every $R$-module has a strongly prime submodule.

(3) Every $R$-module has a prime submodule.

(4) Every $R$-module has a semiprime submodule.
\end{corollary}
\begin{proof}
(1)$\Longrightarrow$(2) and (2)$\Longrightarrow$(3) follow
easily from  Proposition 1.1.\\
(3)$\Longrightarrow$(4) is trivial and (4)$\Longrightarrow$(1)
follows  from Behboodi et al. (2004, Theorem 3.9).
\end{proof}

Next, we observe that strongly prime submodules behave naturally
under localization.
\begin{theorem} Let $M$ be an $R$-module, and let $U$ be
a multiplicatively closed subset of $R$. Then

$${\OS}_{U^{-1}R}U^{-1}M=\{U^{-1}P|P\in{\OS}_{R}M\,{\mbox{
and}}\,\,U^{-1}P\neq U^{-1}M\}.$$

If, moreover, $M$ is finitely generated, then

$${\OS}_{U^{-1}R}U^{-1}M=\{U^{-1}P|P\in{\OS}_{R}M\,{\mbox{
and}}\,\,(P:M)\cap U=\emptyset\}.$$
\end{theorem}
\begin{proof}
First assume that $P\in{\OS}_{R}M$ and $U^{-1}P\neq U^{-1}M$. We
show that $U^{-1}P\in{\OS}_{U^{-1}R}U^{-1}M$. Let
$I_{x_{1}/u_{1}}^{U^{-1}P}x_{2}/u_{2}\subseteq U^{-1}P$, where
$x_{1}/u_{1},x_{2}/u_{2}\in U^{-1}M$. We claim that
$I_{x_1}^{P}x_{2}\subseteq P$. If $r\in I_{x_1}^{P}$, then
$rM\subseteq P+Rx_{1}$ and hence
$$(r/1)U^{-1}M\subseteq
U^{-1}P+U^{-1}R(x_{1}/1)=U^{-1}P+U^{-1}R(x_{1}/u_{1}).$$
Therefore $(r/1)(x_{2}/u_{2})\subseteq U^{-1}P$ and so there
exist $p\in P$ and $v_{1}, v_{2}\in U$ such that
$v_{2}(v_{1}rx_{2}-pu_{2})=0$. This implies that
$(v_{1}v_{2})rx_{2}\in P$.  On the other hand, it is easy to see
that $U^{-1}P\neq U^{-1}M$ implies $(P:M)\cap U=\emptyset$. So we
have $rx_{2}\in P$. Thus $I_{x_1}^{P}x_{2}\subseteq P$. It
follows that $x_{1}\in P$ or $x_{2}\in P$ and hence
$(x_{1}/u_{1})\in U^{-1}P$ or
$(x_{2}/u_{2})\in U^{-1}P$, as desired.\\
Now let $Q\in{\OS}_{U^{-1}R}U^{-1}M$. Set $P=\{x\in
M|x/1\in Q\}$. It is easy to see that $Q=U^{-1}P$ and $P\in{\OS}_{R}M$ and thus we are done.\\
For the second assertion, it is enough to show that $(P:M)\cap
U=\emptyset$ implies that $U^{-1}P\neq U^{-1}M$. Suppose on the
contrary that $U^{-1}P= U^{-1}M$.  Since $M$ is finitely
generated, we may assume that there exist elements
$x_{1},x_{2},\ldots,x_{n}\in M$ that generate $M$. For each $1\leq
i\leq n$ there exist $u_{i},v_{i}\in U$ and $p_{i}\in P$ such that
$v_{i}(u_{i}x_{i}-p_{i})=0$. If $t=v_{1}\ldots v_{n}u_{1}\ldots
u_{n}$, then $t\in(P:M)\cap U$, which is a contradiction.
\end{proof}

The following is an immediate consequence of Theorem 1.5.
\begin{corollary}
Let $M$ be a  finitely generated $R$-module and $U$ be a
multiplicatively closed subset of $R$. Then there is a bijective
inclusion-preserving mapping
\begin{eqnarray*}
\{P\in{\OS}_{R}M| (P:M)\cap U=\emptyset\}&\longrightarrow&{\OS}_{U^{-1}R}U^{-1}M\\
P&\longmapsto& U^{-1}P.
 \end{eqnarray*}
whose inverse is also inclusion-preserving.
\end{corollary}

Let $N$ be a  proper submodule of $M$. The {\it strongly prime
radical} of $N$ in $M$, denoted ${\SR}(N)$, is defined to be the
intersection of all strongly prime submodules of $M$ containing
$N$. If there is no strongly prime submodule containing $N$, then
we put ${\SR}(N)=M$.

We conclude this section with a good justification for the study
of strongly prime submodules. In fact, as it mentioned in the
introduction it is not true that every semiprime submodule of an
$R$-module $M$ is an intersection of prime submodules, see
Jenkins and Smith (1992), but our next theorem shows that as in
the ideal case, this is true for strongly semiprime submodules.
\begin{theorem}
Let $C$ be a strongly semiprime submodule of an $R$-module $M$.
Then $C$ is an intersection of some strongly prime submodules of
$M$.
\end{theorem}
\begin{proof}
It is enough to show that ${\SR}(C)\subseteq C$. Let $x\in
M\setminus C$. We define $T=\{x_{0},x_{1},\ldots\}$ inductively
as follows: $x_{0}=x$, $x_{1}\in I_{x_{0}}^{C}x_{0}\setminus C$,
$x_{2}\in I_{x_{1}}^{C}x_{1}\setminus C$,$\ldots$, etc. Set
$$\Omega=\{K\leq M\,|\, C\subseteq K,\,\,K\cap T=\emptyset\}.$$
$\Omega\neq\emptyset$, since $C\in\Omega$. Then by Zorn's lemma
$\Omega$ has a maximal element, say $P$. We claim that $P$ is a
strongly prime submodule of $M$. Suppose on the contrary that
$x,y\in M\setminus P$ and $I_{x}^{P}y\subseteq P$. Since
$x,y\not\in P$, we have $(P+Rx)\cap T\neq\emptyset$ and
$(P+Ry)\cap T\neq\emptyset$. So there exist $r_{1},r_{2}\in R$
and $p_{1},p_{2}\in P$ and $x_{i},x_{j}\in T$ such that
$p_{1}+r_{1}x=x_{i}$ and $p_{2}+r_{2}y=x_{j}$. We have
$$I_{x_{i}}^{C}(x_{j}-p_{2})\subseteq I_{x_{i}}^{P}(x_{j}-p_{2})=I_{r_{1}x+p_{1}}^{P}
(x_{j}-p_{2})\subseteq I_{x}^{P}(x_{j}-p_{2})=I_{x}^{P}
r_{2}y\subseteq I_{x}^{P}y\subseteq P$$ If $i\geq j$, then there
exists $a\in R$ such that $x_{i}=ax_{j}$ and hence
$$I_{x_{i}}^{C}(x_{i}-p_{2})=I_{x_{i}}^{C}(ax_{j}-p_{2})\subseteq
I_{x_{i}}^{C}(x_{j}-p_{2})\subseteq P.$$ Since $x_{i+1}\in
I_{x_{i}}^{C}x_{i}$, we have $x_{i+1}\in P$, which is a
contradiction. If $i<j$, then there exists $b\in R$ such that
$x_{j}=bx_{i}$ and hence
$$I_{x_{j}}^{C}(x_{j}-p_{2})=I_{bx_{i}}^{C}(x_{j}-p_{2})\subseteq
I_{x_{i}}^{C}(x_{j}-p_{2})\subseteq P.$$ Since $x_{j+1}\in
I_{x_{j}}^{C}x_{j}$, we  have $x_{j+1}\in P$, which is  again a
contradiction. Therefore $P$ is a strongly prime and hence
$x\not\in{\SR}(C)$ and the proof is complete.
\end{proof}

\section{A Generalized Principal Ideal Theorem for Modules}
The Generalized Principal Ideal Theorem (GPIT) states that if $R$
is a Noetherian rings and ${\fp}$ is a minimal prime ideal of an
ideal $(a_{1},\ldots,a_{n})$ generated by $n$ elements of $R$,
then ${\mbox{ht}}{\fp}\leq n$. Consequently,
${\mbox{ht}}(a_{1},\ldots, a_{n})\leq n$, where for an ideal $I$
of $R$, ${\mbox{ht}}I$ denotes the height of $I$.

Krull proved this theorem by induction on $n$. The case $n=1$ is
then the hardest part of the proof. Krull called the $n=1$ case
the Principal Ideal Theorem (PIT).

 \begin{remark}
 The PIT is one of the cornerstones of dimension
  theory for Noetherian rings, see Eisenbud (1995, Theorem 10.1).
   Indeed, Kaplansky (1974, page 104)
   call it ``the most important single theorem in
 the theory of Noetherian rings''.
\end{remark}
  It is natural to ask if the GPIT can be
 extended to modules. Nishitani (1998), has proved that the GPIT holds for modules.
  The aim of this section is to give
  an alternative generalization of GPIT to modules. For this purpose we need to define some
 notions.\\

Let $P$ be a strongly prime submodule of $M$. We shall say that
$P$ is  {\it strongly minimal prime} over a submodule $N$ of $M$,
if $N\subseteq P$ and there does not exist a strongly prime
submodule $L$ of $M$ such that $N\subseteq L\subset P$.
\begin{definition}
(1) Let $P$ be a strongly prime submodule of $M$. The {\it {
strong height}} of $P$, denoted ${\height}_{R}{P}$, is defined by
$${\height}_{R}P={\sup}\{n|\exists\; {P}_{0},{P}_{1},\ldots,{P}_{n}\in{\OS}_{R}M\;\;
 {\mbox{such that}}\;\; {P}_{0}\subset{P}_{1}\subset\cdots\subset{P}_{n}={P}\}.$$
\noindent (2) Let ${N}$ be a proper submodule of an $R$-module
$M$. The {\it {strong height}} of ${N}$, denoted
${\height}_{R}{N}$, is defined by
$${\height}_{R}{N}={\min}\{{\height}_{R}{P}| {P}\in{\OS}_{R}{M},\; P\;
{\mbox{is strongly minimal prime over}}\;N\} .$$
\end{definition}

\begin{theorem}
Let $R$ be a ring and $M$ be a Noetherian flat $R$-module. Let $N$
be a proper submodule of $M$ generated by $n$ elements
$x_{1},\ldots,x_{n}\in M$. Then ${\height}_{R}N\leq n.$
\end{theorem}
\begin{proof}
 Replacing $R/(0:M)$ by $R$, we can assume that $R$ is a
 Noetherian ring. Let ${\height}_{R}N=\ell$. Then there is a submodule $P$ of $M$ such that
 $P$ is strongly minimal prime over $N$ and ${\height}_{R}P=\ell$. Let  ${\fp}=(P:M)$ and
  $U=R\setminus{\fp}$. By
  Corollary 1.6, ${\height}_{R}N={\height}_{U^{-1}R}U^{-1}N$. Thus replacing $U^{-1}R$ by
  $R$, we can assume that $R$ is a Noetherian local ring with
  maximal ideal ${\fp}$. Because $M$ is a flat module over a local
  ring, it is free with finite rank, say $m$. Since $M/P$ is an
  $R/{\fp}$-vector space and $(0)$ is a strongly prime submodule
  of $M/P$, by Proposition 1.3, we have ${\dim}_{R/{\fp}}M/P=1$. Hence there exists a basis
   $\{e_{1},e_{2},\ldots,e_{m}\}$ for $M$ such that $e_{1},e_{2},\ldots,e_{m-1}\in P$
  and $e_{m}\not\in P$. We have
  $P=Re_{1}+Re_{2}+\ldots+Re_{m}+{\fp}e_{m}$. There are elements
  $a_{1j},a_{2j},\ldots,a_{m-1j}\in R$ and $a_{mj}\in{\fp}$ such that
  $x_{j}=a_{1j}e_{1}+a_{2j}e_{2}+\ldots+a_{mj}e_{m}$. Let ${\fq}$
  be a minimal prime ideal over an ideal
  $(a_{m1},a_{m2},\ldots,a_{mn})$ and $Q$ denotes the submodule
    $Re_{1}+Re_{2}+\ldots+Re_{m}+{\fq}e_{m}$. Since
    $M/Q\cong R/{\fq}$, $Q$ is a strongly prime submodule
    and hence $P=Q$, by the minimality of $P$. Hence ${\fp}={\fq}$
    holds and so ${\fp}$ is a minimal prime over an ideal generating
    by $n$ elements. Since  ${\height}_{R}P=\ell$, we  can consider the
    following chain of distinct strongly prime submodules of $M$
    $$P_{0}\subset P_{1}\subset\ldots\subset P_{\ell}=P.$$
    We claim that the above chain induces a chain
$$(P_{0}:M)\subset(P_{1}:M)\subset\ldots\subset (P_{\ell}:M)={\fp}$$
of distinct prime ideals of $R$. It is enough to show that
$(P_{0}:M)\subset(P_{1}:M)$. The containment
$(P_{0}:M)\subseteq(P_{1}:M)$ is always true. Suppose that
$(P_{0}:M)=(P_{1}:M)$. Then for any $x\in P_{1}\setminus P_{0}$
and any $y\in M$, we have

$$I_{x}^{P_{0}}y\subseteq I_{x}^{P_{1}}y=\{r\in R|rM\subseteq
 P_{1}\}y=(P_{1}:M)y=(P_{0}:M)y\subseteq P_{0}.$$
Since $P_{0}$ is strongly prime and $x\not\in P_{0}$, we have
$y\in P_{0}$ and hence $P_{0}=M$ which is a contradiction. Thus
$(P_{0}:M)\subset(P_{1}:M)$. Now by the GPIT for rings, we have
$\ell\leq {\mbox{ht}}_{R}{\fp}\leq n$. This completes the proof.
\end{proof}

\section*{Acknowledgments} I thank Javad
Asadollahi for his suggestions and comments. I also thank the
referee for many careful comments, and Sharekord University for
the financial support.

\end{document}